\documentclass[10pt]{article}

\usepackage[margin=1.5in]{geometry}

\usepackage{amsmath}
\usepackage{amsthm}
\usepackage{amsfonts}
\usepackage{makecell}
\usepackage{booktabs}

\usepackage{marvosym}
\usepackage{MnSymbol}

\usepackage{listings}
\newtheorem{theorem}{Theorem}[section]
\newtheorem{lemma}[theorem]{Lemma}
\newtheorem{corollary}[theorem]{Corollary}

\theoremstyle{definition}
\newtheorem{definition}[theorem]{Definition}

\usepackage{scalerel}
\usepackage{stackengine,wasysym}

\usepackage{array}
\newcolumntype{L}[1]{>{\raggedright\let\newline\\\arraybackslash\hspace{0pt}}m{#1}}
\newcolumntype{C}[1]{>{\centering\let\newline\\\arraybackslash\hspace{0pt}}m{#1}}
\newcolumntype{R}[1]{>{\raggedleft\let\newline\\\arraybackslash\hspace{0pt}}m{#1}}

\newcommand\reallywidetilde[1]{\ThisStyle{%
  \setbox0=\hbox{$\SavedStyle#1$}
  \setbox1=\hbox{$x_1$}
  \stackengine{-.1\LMpt}{$\SavedStyle#1$}{%
    \stretchto{\scaleto{\SavedStyle\mkern.2mu\AC}{.5150\wd0}}{.8\ht1}%
  }{O}{c}{F}{T}{S}%
}}

\newcommand{\rev}[1]{{#1}^R}
\newcommand{\drev}[1]{\overset{R}{\reallywidetilde{#1}}}
\newcommand{\revant}[1]{\reallywidetilde{#1}}

\newcommand{\ediv}{\mathbin{|_e}}
\renewcommand{\div}{\mathbin{|}}
\newcommand{\Simp}{\mathrm{Simp}}
\newcommand{\irr}{\mathrm{irr}}
\newcommand{\rest}[2]{{#1}_{({#2})}}
\newcommand{\restle}[2]{{#1}_{(\leq {#2})}}
\newcommand{\ind}{\mbox{ind}}

\usepackage{enumerate}
\usepackage{enumitem}
\setenumerate{label={\normalfont (\alph*)}}

\newenvironment{case}{%
 \let\olditem\item%
 \renewcommand\item[1][]{\olditem {##1} \\}%
 \begin{enumerate}[label=\textbf{Case \arabic*:},itemindent=*,leftmargin=0em]}{\end{enumerate}%
 }

\title{Avoidance bases for formulas with reversal}

\author{James Currie, Lucas Mol, Narad Rampersad}

\date{April 11, 2018}

\begin{document}

\maketitle

\begin{abstract}
In the interest of studying formulas with reversal of high avoidability index, we find $n$-avoidance bases for formulas with reversal for $n\in\{1,2,3\}$.  We demonstrate that there is a unique formula with reversal in each of these three bases of highest avoidability index $n+2$; these formulas are $xx$, $xyx\cdot \rev{y}$, and $xyzx\cdot \rev{y}\cdot \rev{z}$, which belong to an infinite family of formulas with reversal that has been the subject of recent study by the authors.


\end{abstract}

\section{Preliminaries}


Let $\Sigma$ be a set of letters called \textit{variables}.  A \textit{pattern} $p$ over $\Sigma$ is a finite word over $\Sigma$.  A \textit{formula} $\phi$ over $\Sigma$ is a finite set of patterns over $\Sigma.$  We usually use dot notation to denote formulas; that is, for $p_1,\dots,p_n\in \Sigma^*$ we let
\[
p_1\cdot p_2\cdot \ldots\cdot p_n=\{p_1,p_2,\dots,p_n\}.
\]
We sometimes refer to formulas as \emph{classical formulas} to differentiate them from formulas with reversal.

For an alphabet $\Sigma,$ define the \textit{reversed alphabet} $\Sigma^R=\{\rev{x}\colon\ x\in \Sigma\},$ where $\rev{x}$ denotes the \textit{reversal} or \textit{mirror image} of variable $x.$  A \textit{pattern with reversal} over $\Sigma$ is a finite word over alphabet $\Sigma\cup\Sigma^R.$  A \textit{formula with reversal} over $\Sigma$ is a finite set of words over $\Sigma\cup\Sigma^R,$ i.e.\ a finite set of patterns with reversal over $\Sigma.$  The elements of a formula (with reversal) $\phi$ are called the \textit{fragments} of $\phi.$

For words over any alphabet $A,$ we denote by $\revant{-}$ the reversal antimorphism; if $a_1,a_2,\dots, a_n\in A$, then
\[
\reallywidetilde{a_1a_2\dots a_n}=a_na_{n-1}\dots a_1.
\]
We say that a morphism $f:(\Sigma\cup \Sigma^R)^*\rightarrow A^*$ \textit{respects reversal} if $f(\rev{x})=\revant{f(x)}$ for all variables $x\in \Sigma.$  Note that any morphism $f:\Sigma^*\rightarrow A^*$ extends uniquely to a morphism from $(\Sigma\cup \Sigma^R)^*$ that respects reversal.

Let $p$ be a pattern (with reversal).  An \textit{instance} of $p$ is the image of $p$ under some non-erasing morphism (respecting reversal).  A word $w$ \textit{avoids} $p$ if no factor of $w$ is an instance of $p.$  Let $\phi$ be a formula (with reversal).  We say that $\phi$ \textit{occurs} in $w$ if there is a non-erasing morphism $h$ (which respects reversal) such that the $h$-image of every fragment of $\phi$ is a factor of $w.$  In this case we say that $\phi$ occurs in $w$ \textit{through $h$}, or that $w$ \textit{encounters} $\phi$ through $h.$  If $\phi$ does not occur in $w$ then we say that $w$ \textit{avoids} $\phi.$  For any $k\geq 1$, let $A_k$ denote an alphabet of size $k$.  We say that $\phi$ is \textit{$k$-avoidable} if there are infinitely many words over $A_k$ which avoid $\phi$. Equivalently, $\phi$ is $k$-avoidable if there is an $\omega$-word $\mathbf{w}$ over $A_k$ such that every finite prefix of $\mathbf{w}$ avoids $\phi$ (in this case we say that $\mathbf{w}$ avoids $\phi$).  If $\phi$ is $k$-avoidable for some $k$ then we say that $\phi$ is \emph{avoidable}; the \emph{avoidability index} of $\phi$, denoted $\ind(\phi),$ is the smallest integer $k$ such that $\phi$ is $k$-avoidable.  If $\phi$ is not $k$-avoidable for any natural number $k$, then we say that $\phi$ is \emph{unavoidable}, and we define $\ind(\phi)=\infty$.

Formulas were introduced by Cassaigne \cite{CassaigneThesis}, and it was shown that every formula corresponds in a natural way to a pattern of the same \textit{avoidability index} (see \cite{CassaigneThesis} or \cite{ClarkThesis} for details).  Essentially, this means that formulas are a natural generalization of patterns in the context of avoidability.

In order to define divisibility of formulas with reversal, we require a second notion of reversal in $(\Sigma\cup\rev{\Sigma})^*$ which not only reverses the letters of a word in $(\Sigma\cup \rev{\Sigma})^*$, but also swaps $x$ with $\rev{x}$ for all $x\in\Sigma.$  For $x_1,x_2,\dots,x_n\in\Sigma\cup \rev{\Sigma},$ we define \textit{d-reversal} $\drev{-}$ by
\[
\drev{x_1x_2\dots x_n}=\revant{\rev{x}_1\rev{x}_{2}\dots\rev{x}_n}=\rev{x}_n\rev{x}_{n-1}\dots\rev{x}_1,
\]
where $\rev{(\rev{x})}=x$ for all $x\in\Sigma$ (note that the d stands for division).  A morphism $h:(\Sigma\cup\rev{\Sigma})^*\rightarrow(\Sigma\cup \rev{\Sigma})^*$ \textit{respects d-reversal} if 
\[
h(\rev{x})=\drev{h(x)}
\] 
for all $x\in\Sigma.$  Note that any morphism $f:\Sigma^*\rightarrow (\Sigma\cup\rev{\Sigma})^*$ extends uniquely to a morphism from $(\Sigma\cup \Sigma^R)^*$ that respects d-reversal.

We say that a pattern (with reversal) $p$ is a \textit{factor} of formula (with reversal) $\phi$ if $p$ is a factor of some fragment of $\phi$.  Let $\phi$ and $\psi$ be formulas with reversal over $\Sigma$.  We say that $\phi$ \textit{divides} $\psi$, written $\phi \div \psi$, if there is a non-erasing morphism $h:(\Sigma\cup \rev{\Sigma})^*\rightarrow (\Sigma\cup \rev{\Sigma})^*$ which respects d-reversal such that the $h$-image of every fragment of $\phi$ is a factor of $\psi.$  We say that $\phi$ \textit{e-divides} $\psi$, written $\phi\ediv\psi$, if there is some injective morphism respecting $d$-reversal $h$ having $|h(x)|=1$ for all $x\in\Sigma\cup \rev{\Sigma}$ such that $\phi\div\psi$ through $h$.   We say that $\phi$ and $\psi$ are \textit{equivalent} (resp.\ \textit{e-equivalent}) if they divide (resp.\ e-divide) one another.

For example, the formula with reversal $xyx\cdot \rev{y}$ divides $xyzxyz\cdot \rev{z}\rev{y}\rev{z}$ through the morphism respecting d-reversal $h$ defined by $h(x)=x$ and $h(y)=yz$.  The formula with reversal $xyx\cdot \rev{y}$ e-divides $y\rev{x}y\cdot x$ through the morphism respecting d-reversal $g$ defined by $g(x)=y$ and $g(y)=\rev{x}$.  In fact, since $y\rev{x}y\cdot x$ e-divides $xyx\cdot\rev{y}$ as well, $xyx\cdot \rev{y}$ and $y\rev{x}y\cdot x$ are e-equivalent.

It is straightforward to show that if $\phi$ divides $\psi$ through morphism respecting d-reversal $h$ and $\psi$ occurs in a word $w$ through morphism respecting reversal $f$, then $f\circ h$ respects reversal and $\phi$ occurs in $w$ through $f\circ h.$  Thus if $\psi$ is unavoidable and $\phi$ divides $\psi$, then $\phi$ is unavoidable as well.  On the other hand, if $\phi$ is avoidable and $\phi$ divides $\psi$, then $\psi$ is avoidable as well, and $\ind(\phi)\geq \ind(\psi)$.

For any natural number $n$, we let $\rest{\phi}{n}$ denote the formula (with reversal) whose fragments are the factors of $\phi$ of length $n$.  We let $\restle{\phi}{n}$ denote the formula (with reversal) whose fragments are the factors of $\phi$ of length at most $n.$    For example, if $\phi=xyzx\cdot xz\cdot \rev{y}$, then $\rest{\phi}{2}=xy\cdot yz\cdot zx\cdot xz$, and $\restle{\phi}{2}=xy\cdot yz\cdot zx\cdot xz\cdot x\cdot y\cdot z\cdot \rev{y}$.

A fragment $p$ of a formula with reversal $\phi$ is called \textit{redundant} if it is a factor of another fragment $q$ of $\phi$ (where $q\neq p$).  A formula with reversal $\phi$ is called \textit{irredundant} if it has no redundant fragments.  Every formula with reversal $\phi$ is e-equivalent to the irredundant formula $\irr(\phi)$ obtained by discarding the redundant fragments.

Let $p$ be a pattern with reversal over $\Sigma.$  The \textit{flattening} of $p$, denoted $p^\flat,$ is the image of $p$ under the morphism defined by $x\mapsto x$ and $\rev{x}\mapsto x$ for all $x\in\Sigma$.  We say that $p$ \textit{flattens} to $p^\flat.$  The \textit{flattening} of a formula with reversal $\phi,$ denoted $\phi^\flat,$ is the set of flattenings of all fragments of $\phi,$ i.e.\ $\phi^\flat=\{p^\flat\colon\ p\in\phi\}.$  Again, we say that $\phi$ \textit{flattens} to $\phi^\flat.$  It was shown in \cite{CMR2017} that if $\phi^\flat$ is avoidable, then $\phi$ is avoidable.  It follows that if $\phi$ is a pattern with reversal of length at least $2^n$ over an alphabet $\Sigma$ of size $n$, then $\phi$ is avoidable.

For a variable $x\in\Sigma,$ we let $x^\sharp$ denote the set $\{x,\rev{x}\}$ containing $x$ and its mirror image.  For sets $X$ and $Y$, we let
\[
XY=\{xy\colon\ x\in X\mbox{ and } y\in Y\},
\]
so that 
\[
x^\sharp y^\sharp=\{xy,x\rev{y},\rev{x}y,\rev{x}\rev{y}\}=xy\cdot x\rev{y}\cdot \rev{x}y\cdot \rev{x}\rev{y},
\] 
for example.  We often write the set containing a single word $w$ as simply $w$ instead of $\{w\}$ when using this notation.  For example,
\[
x^\sharp yx^\sharp=xyx\cdot xy\rev{x}\cdot \rev{x}yx\cdot \rev{x}y\rev{x}.
\]

For a formula with reversal $\phi$ over $\Sigma$, a variable $x\in\Sigma$ is called \emph{two-way} in $\phi$ if both $x$ and $\rev{x}$ are factors of $\phi$; otherwise, $x$ is called \emph{one-way} in $\phi$.

Finally, to describe a morphism $f\colon\{0,1,\dots,m\}^*\rightarrow\{0,1,\dots,n\}^*,$ we use the shorthand $f=f(0)/f(1)/\dots/f(m).$  For example, $g=01/2/031/3$ denotes the morphism $g\colon\{0,1,2,3\}^*\rightarrow \{0,1,2,3\}^*$ defined by $g(0)=01,$ $g(1)=2$, $g(2)=031,$ and $g(3)=3$.

\section{Introduction}

Formulas with reversal are a relatively new object of study in combinatorics on words, but they have received considerable attention due to some interesting and surprising results.  The number of binary words avoiding the pattern with reversal $xx\rev{x}$ was shown to be intermediate between polynomial and exponential~\cite{CurrieRampersad2016}, and this is the first time that such an intermediate growth rate has been observed in the context of pattern avoidance.  A similar growth rate was observed for binary words avoiding $x\rev{x}x$ soon afterwards~\cite{CurrieRampersad2015}.

Recently, the authors have found an infinite family of formulas with reversal of high avoidability index~\cite{CMR2016}, and have studied avoidability of formulas with reversal in general~\cite{CMR2017}.  For $n\geq 1$, define
\begin{align*}
\psi_n&=xy_1y_2\dots y_{n}x\cdot \rev{y_1}\cdot \rev{y_2}\cdot \ldots\cdot \rev{y_n}.
\end{align*}
In \cite{CMR2016}, it is shown that $\ind(\psi_1)=4,$ $\ind(\psi_n)=5$ for $n\in\{2,3,6\},$ $\ind(\psi_n)\geq 5$ for $n\in\{4,5\}$, and $\ind(\psi_n)\geq 4$ for all $n\geq 7.$  The constant general upper bound $\ind(\psi_n)\leq 5+(n\bmod 3)$ is also established.   We suspect that $\ind(\psi_n)=5$ for all $n\geq 2$.  Here, we extend this family in a natural way by defining $\psi_0=xx$.

In \cite{CMR2017}, the authors characterize the unavoidable formulas with reversal having at most two one-way variables.  It follows from this result and the well-known characterization of classical unavoidable formulas \cite{Zimin1984}, that if $\phi$ is a formula with reversal on at most two letters, then $\phi$ is unavoidable if and only if $\phi$ divides some formula from
\[
\mathcal{Z}_2=\{x^\sharp y^\sharp, x^\sharp y x^\sharp\},
\]
and that if $\phi$ is a formula with reversal on at most three letters, then $\phi$ is unavoidable if and only if $\phi$ divides some formula from
\[
\mathcal{Z}_3=\{x^\sharp y^\sharp z^\sharp, x^\sharp y^\sharp zx^\sharp y^\sharp, x^\sharp yx^\sharp z x^\sharp y x^\sharp\}.
\]

We note that the avoidability index of every pattern with reversal on at most two variables has been determined~\cite{CurrieLafrance2016}.  Here, we are interested in determining the avoidable formulas with reversal on at most three variables of highest avoidability index.  To this end, we define \emph{avoidance bases} for formulas with reversal, extending the idea of Clark \cite{ClarkThesis} for regular formulas.  An $n$-avoidance basis is a collection of ``minimally avoidable'' formulas on $n$ variables - those that are not properly divisible by any other avoidable formulas.  The definition and theory of avoidance bases for formulas with reversal is given in Section \ref{Bases}.  In Section \ref{FindBases}, we find $n$-avoidance bases for $n\in\{1,2,3\}$.  It follows from known results that $xx$ (that is, $\psi_0$) is the unique element of highest avoidability index $3$ in our $1$-avoidance basis, and that $xyx\cdot \rev{y}$ (which is equivalent to $\psi_1$) is the unique element of highest avoidability index $4$ in our $2$-avoidance basis.  The remainder of the article is committed to showing that $xyzx\cdot \rev{y}\cdot \rev{z}$ (equivalent to $\psi_2$) is the unique element of highest avoidability index $5$ in our $3$-avoidance basis.  This leads us to wonder whether $\psi_n$ has highest avoidability index among all formulas with reversal on $n+1$ variables for all $n$.  While it is tempting to conjecture that it is true, we suspect that it is not, given the constant general upper bound on $\ind(\psi_n)$ given in \cite{CMR2016}.

\section{Avoidance Bases}\label{Bases}

Since $\phi\div \psi$ implies $\ind(\phi)\geq \ind(\psi)$, the formulas with reversal with the highest avoidability indices should be those not divisible by any other (non-equivalent) avoidable formulas.  In a similar situation for regular formulas, Clark \cite{ClarkThesis} defined an \emph{avoidance basis}, which is a collection that describes all such ``minimally avoidable'' formulas.  We extend this idea to formulas with reversal.

\begin{definition}\label{BasisDef}
Fix an alphabet $\Sigma_n$ of order $n$.  A set $\Phi$ of formulas (with reversal) over $\Sigma_n$ is called an \textit{$n$-avoidance basis (for formulas with reversal)} if both of the following conditions are satisfied.
\begin{itemize}
\item For any avoidable formula (with reversal) $\psi$ over $\Sigma_n$, there is some $\phi\in \Phi$ such that $\phi\ediv \psi$, and
\item For any $\phi_1,\phi_2\in \Phi$, if $\phi_1\ediv \phi_2$, then $\phi_1=\phi_2.$
\end{itemize}
A formula (with reversal) $\phi$ is called \textit{$n$-minimal} if it belongs to an $n$-avoidance basis (for formulas with reversal).  We say that $f$ is {minimal} if it is $n$-minimal for some $n.$
\end{definition}

Much of the theory concerning $n$-avoidance bases for classical formulas translates directly to the situation for formulas with reversal.  In particular, $n$-avoidance bases for formulas with reversal exist for each $n$, and there is a nice characterization of minimal formulas with reversal.  The most important results are stated below.  We omit the proofs as they are analogous to those found in \cite{ClarkThesis} for classical formulas.  The only minor difference results from the fact that in \cite{ClarkThesis}, simplifications are not defined for fragments of length $1.$  Defining simplifications as follows makes the theory work nicely for formulas with reversal.

\begin{definition}
Let $\phi$ be a formula with reversal with fragment $q\neq \varepsilon$.  The \emph{$q$-simplification of $\phi$}, denoted $\Simp(\phi,q)$, is given by 
\[
\Simp(\phi,q)=
\begin{cases}
\phi-\{q\} \mbox{ if $|q|=1$, and}\\
(\phi-\{q\})\cdot p\cdot s \mbox{ if $|q|>1$,}
\end{cases}
\]  
where $p$ is the length $|q|-1$ prefix of $q$ and $s$ is the length $|q|-1$ suffix of $q$.  If $\psi=\Simp(\phi,q)$ for some $q\in \phi$, then $\psi$ is called a \textit{simplification} of $\phi$.
\end{definition}

\begin{theorem}[c.f. Clark \cite{ClarkThesis}]\label{BasisTheory}\ 
\begin{enumerate}
\item For every natural number $n$, there exists an $n$-avoidance basis for formulas with reversal.

%

\item Let $\Phi$ be an $n$-avoidance basis for formulas with reversal and let $\Psi$ be a set of avoidable formulas with reversal over an alphabet $\Sigma$ of order $n$.  Then $\Psi$ is an $n$-avoidance basis for formulas with reversal if and only if there exists a bijection $f\colon \Phi\rightarrow \Psi$ such that $\phi$ and $f(\phi)$ are e-equivalent for all $\phi\in\Phi.$

\item If $\phi$ is $n$-minimal, then $\phi$ is $(n+1)$-minimal.

\item \label{MinimalChar} Let $\phi$ be an irredundant formula with reversal.  Then $\phi$ is minimal if and only if both $\phi$ is avoidable and every simplification of $\phi$ is unavoidable.\hfill \qed
\end{enumerate}
\end{theorem}
It is also straightforward to show that any $n$-avoidance basis for classical formulas is a subset of some $n$-avoidance basis for formulas with reversal.

In \cite{ClarkThesis}, Clark found an $n$-avoidance basis for classical formulas for each $n\in\{1,2,3\}$, and demonstrated that every $3$-minimal formula has avoidability index at most $4$.  This means that no classical avoidable formula on at most three variables has avoidability index greater than $4$.  Recently, the exact avoidability index of every $3$-minimal formula has been determined~\cite{CircularFormulas}.

\section{Finding avoidance bases for formulas with reversal}\label{FindBases}

In this section, we find $n$-avoidance bases for formulas with reversal for $n\in\{1,2,3\}$.  A $1$-avoidance basis $\Phi_1$ is given by $\{xx,x\rev{x}\},$ and this is easily verified by inspection.  A $2$-avoidance basis $\Phi_2$ and a $3$-avoidance basis $\Phi_3$ for formulas with reversal are shown in Table~\ref{Avoidance2} and Table~\ref{Avoidance3}, respectively.  What is known about the avoidability index of each formula with reversal in these bases is also included in the table, and an infinite word avoiding each formula with reversal on as few letters as is known to be possible is given.  Many of these infinite words are periodic, and we omit the proofs that they avoid the corresponding formulas with reversal as they are straightforward.  For the nonperiodic infinite words given in each table, we provide a reference, or point the reader to the relevant section of this article where the avoidance is proven.

It is straightforward to verify computationally that each of the given formulas in $\Phi_2$ and $\Phi_3$ is minimal using Theorem~\ref{BasisTheory}\ref{MinimalChar} and the known characterization of avoidable formulas with reversal on at most three variables.  However, this only tells us that the given formulas with reversal belong to some avoidance basis, not that they make up an avoidance basis together.  To verify this stronger fact, we return to the definition of avoidance basis for formulas with reversal.  While the second condition of the definition can be verified directly with a straightforward check, more work is required to demonstrate the first condition.  In order to show that every avoidable formula with reversal on two (or three) variables is e-divisible by some element of $\Phi_2$ ($\Phi_3$, respectively), we eliminate all but a finite number of avoidable formulas, and then complete an exhaustive check using a computer.  Throughout, we let $\Sigma_2=\{x,y\}$, and $\Sigma_3=\{x,y,z\}$, and we work exclusively with formulas with reversal over these alphabets.

\begin{table}[h]
\begin{center}
\def\arraystretch{1.5}
\begin{tabular}{L{2.5cm} C{1cm} L{7cm}}\toprule
Formula & Index & Avoidance Properties\\\toprule

$x\rev{x}$ & $2$ & Avoided by $(01)^\omega$.\\\midrule

$xx$ & $3$ & Avoided by $f^\omega(0)$, where $f=012/02/1$ (This word is also the unique preimage of the Thue-Morse word under the morphism $\delta:\{0,1,2\}^*\rightarrow \{0,1\}^*$ defined by $\delta=011/01/0$; see Section 2.3 of~\cite{Lothaire1997} for details.); longest word on two letters has length $3$.  \\\midrule

$xyx\cdot yxy$ & $3$ & Avoided by $f_2(f_1^\omega(0)),$ where $f_1=01/02/32/31$ and $f_2=01/02/12/21$ \cite{CassaigneThesis}; longest word on two letters has length $8$.\\\midrule

{$xy\cdot yx\cdot\rev{x}\cdot \rev{y}$ \newline
$xy\cdot y\rev{x}\cdot \rev{y}$\newline
$xy\cdot \rev{y}\rev{x}$}
& $3$ & 
{Avoided by $(012)^\omega$; longest word on two letters has length $2$.} \\\midrule

$xyx\cdot \rev{y}$  
& $4$ & See \cite{CMR2016} for $4$-avoidance; longest word on three letters has length $14$. \\\bottomrule
\end{tabular}
\end{center}
\caption{A $2$-avoidance basis $\Phi_2$ for formulas with reversal.}
\label{Avoidance2}
\end{table}

\begin{theorem}\label{Avoidance2Theorem}
The collection
\[
\Phi_2=\{xx,x\rev{x},xyx\cdot yxy,xy\cdot yx\cdot\rev{x}\cdot \rev{y},xy\cdot y\rev{x}\cdot \rev{y},xy\cdot \rev{y}\rev{x},xyx\cdot \rev{y}\} 
\]
is a $2$-avoidance basis for formulas with reversal.
\end{theorem}

\begin{proof}
The second condition of Definition \ref{BasisDef} is easy to check.  For the first condition, it suffices to show that every avoidable formula with reversal over the alphabet $\Sigma_2$ is divisible by some member of $\Phi_2$.

First of all, we note that there are
\[
4+4^2+4^3+4^4=340
\]
distinct nonempty \emph{patterns} with reversal over $\Sigma_2$ of length at most $4$, corresponding to formulas with reversal with exactly one fragment.  Using the characterization of unavoidable formulas with reversal on two letters, we find that $28$ of these patterns are unavoidable, while the remaining $312$ are avoidable.  We check that some member of $\Phi_2$ e-divides each of these $312$ avoidable patterns.  Since every pattern with reversal of length at least $4$ is avoidable, and every pattern with reversal of length greater than $4$ is e-divisible by each of its factors of length $4$ (through the identity map), we conclude that every avoidable pattern with reversal over $\Sigma_2$ is e-divisible by some member of $\Phi_2$.

It now suffices to show that every avoidable formula with reversal over $\Sigma_2$ whose fragments are all unavoidable is e-divisible by some member of $\Phi_2$.  In fact, it suffices to deal with irredundant formulas since $\phi$ is e-equivalent to $\irr(\phi)$ for any formula with reversal $\phi$.  Further, we only need to consider avoidable formulas where the deletion of any fragment leaves an unavoidable formula, since every avoidable formula with reversal on unavoidable fragments is e-divisible by such a formula with reversal.  But any such formula on $k\geq 2$ fragments can be written
\[
\phi\cdot p,
\]
where $\phi$ is an irredundant unavoidable formula with reversal over $\Sigma_2$ with $k-1$ fragments (all of which must necessarily be unavoidable) and $p$ is an unavoidable pattern with reversal.  

Let $P$ denote the set of all unavoidable patterns with reversal over $\Sigma_2$.  We write an algorithm (see Figure \ref{Algorithm1}) that takes as input the set $U_k$ of all irredundant unavoidable formulas with reversal over $\Sigma_2$ with exactly $k$ fragments (up to e-equivalence), and outputs the set $U_{k+1}$ of all irredundant unavoidable formulas with reversal over $\Sigma_2$ with exactly $k+1$ fragments (up to e-equivalence).  Along the way, we check that every irredundant avoidable formula of the form
\[
\phi\cdot p,
\] 
where $\phi\in U_k$ and $p\in P$, is e-divisible by some member of $\Phi_2$.  We repeatedly apply this algorithm starting at $k=1$ until we find that $U_5$ is empty, at which point we are done.
\end{proof}

\begin{figure}[h]
\begin{lstlisting}
given: $P$, the set of all unavoidable 
patterns with reversal over $\Sigma_2$.

input: $U_k$, the set of all irredundant unavoidable
formulas with reversal over $\Sigma_2$ with 
exactly $k$ fragments (up to e-equivalence).

output: $U_{k+1}$, the set of all irredundant unavoidable
formulas with reversal over $\Sigma_2$ with 
exactly $k+1$ fragments (up to e-equivalence).

$U_{k+1}=\{\}$

for $\phi$ in $U_k$:
  for $p$ in $P$:
    $\phi'=\phi\cdot p$    
    
    if $\phi'$ has a redundant fragment:
      continue 
    
    if $\phi'$ is unavoidable:
      if $\phi'$ is not e-equivalent to any member of $U_{k+1}$:
        add $\phi'$ to $U_{k+1}$
    else:
      check that some element of $\Phi_2$ e-divides $\phi'$


return $U_{k+1}$
\end{lstlisting}
\caption{The algorithm used in the proof of Theorem \ref{Avoidance2Theorem} to check that every avoidable formula over two variables is e-divisible by some element of $\Phi_2$.}
\label{Algorithm1}
\end{figure}

\begin{table}
\begin{center}
\def\arraystretch{1.3}
\begin{tabular}{L{4.4cm} C{1cm} L{7cm}}\toprule
Minimal Formula & Index & Avoidance Properties\\\toprule

{$xy\cdot xz\cdot \rev{y}z$\newline
$x\rev{y}\cdot y\rev{z}\cdot z\rev{x}$}
& $2$ &
Avoided by $(01)^\omega$.\\\midrule

$xyzyx\cdot zyxyz$ & $2$ & See \cite{CircularFormulas}.\\\midrule

{$xyzyx\cdot zyx\rev{y}z$ (and rev.) \newline
$xyzyx\cdot z\rev{y}x\rev{y}z$ \newline
$xyz\rev{y}x\cdot zyx\rev{y}z$ \newline
$xyz\rev{y}x\cdot z\rev{y}xyz$}  
& $[2,4]$ & 
Avoided by $g^\omega(0)$ or its reversal (see Section \ref{xyzyxSection}).  We have found a binary word of length $1000$ which simultaneously avoids all of these formulas. \\\midrule

{$xy\cdot xz\cdot yz\cdot \rev{y}$  \newline
$xy\cdot xz\cdot \rev{y}\rev{z}$ (and rev.)\newline
$xy\cdot xz\cdot y\rev{z}\cdot \rev{y}$ (and rev.) \newline
$xy\cdot \rev{x}\rev{z}\cdot yz\cdot \rev{y}$ \newline
$xy\cdot \rev{x}\rev{z}\cdot \rev{y}z$ }
& $3$ &
Avoided by $(012)^\omega$; longest word on two letters has length $3$.\\\midrule

$xyzx\cdot yzxy\cdot zxyz$ 
& $3$ &
See \cite{CircularFormulas}; longest word on two letters has length $44$.
\\\midrule

{$xyzx\cdot yzxy\cdot zyz$ (and rev.)
} 
& $3$ &
See \cite{CircularFormulas}; longest word on two letters has length $16$. 
\\\midrule

{$xyzx\cdot yzxy\cdot \rev{z}$\newline
$xyzx\cdot y\rev{z}xy$}  
& $[3,4]$ & 
Avoided by $h(g^\omega(0))$ (see Section \ref{xyzxSection}); longest word on two letters has length $8$.  We have found a ternary word of length $1000$ that simultaneously avoids both of these formulas. \\\midrule

$xy\cdot xz\cdot yx\cdot zx\cdot zy$
& $4$ &
Avoided by $\Omega^\omega(0)$ \cite{BMT1989}; longest word on three letters has length $7$.\\\midrule

{$xy\cdot yz\cdot zx\cdot \rev{x}\cdot \rev{y}\cdot \rev{z}$\newline
$xy\cdot yz\cdot z\rev{x}\cdot \rev{y}\cdot\rev{z}$ \newline
$xy\cdot yz\cdot \rev{z}\rev{x}\cdot\rev{y}$} 
& $4$ & 
Avoided by $(0123)^\omega$; longest word on three letters has length $4$.\\\midrule

$xyzx\cdot \rev{y}\cdot \rev{z}$ 
& 5 &
See \cite{CMR2016}; longest word on four letters has length $45$.\\\bottomrule
\end{tabular}
\end{center}
\caption{A $3$-avoidance basis $\Phi_3$ for formulas with reversal, along with the formulas of $\Phi_2$ given in Table \ref{Avoidance2}.  Throughout, $g=01/2/031/3,$ $h=01/12/20/3,$ and $\Omega=01/21/03/23$.  The notation $[a,b]$ indicates that the exact avoidability index is unknown, but that it is between $a$ and $b$ inclusive.}
\label{Avoidance3}
\end{table}

\begin{theorem}\label{Avoidance3Theorem}
The collection $\Phi_3$ is a $3$-avoidance basis for formulas with reversal.
\end{theorem}

\begin{proof}
We use the same technique as in the proof of Theorem \ref{Avoidance2Theorem}.  The second condition of Definition \ref{BasisDef} is easily checked, and it remains to show that every avoidable formula over $\Sigma_3=\{x,y,z\}$ is e-divisible by some element of $\Phi_3$.

We first check that every avoidable pattern with reversal over $\Sigma_3$ is divisible by some element of $\Phi_3$ by exhaustively checking all patterns with reversal over $\Sigma_3$ of length at most $8$ (we can quickly and easily reduce the number of patterns that need to be checked by eliminating any patterns having factors that flatten to squares).  Then, as in the proof of Theorem \ref{Avoidance2Theorem}, it suffices to show that every avoidable formula of the form $\phi\cdot p$ is divisible by some element of $\Phi_3$, where $\phi$ is an irredundant unavoidable formula with reversal over $\Sigma_3$, and $p$ is an unavoidable pattern with reversal over $\Sigma_3$.  We employ an algorithm analogous to that given in Figure \ref{Algorithm1} to do so.  However, the time required to complete the search proved to be too long without some further elimination.

First of all, we break up the formulas with reversal by the number of two-way variables that they contain, and handle each group separately.
\begin{itemize}
\item We show that every avoidable formula with reversal over $\Sigma_3$ in which all three variables are two-way is e-divisible by some avoidable formula with reversal whose fragments have length at most $2$; this is Lemma \ref{ThreeTwoWay}.  So to show that any formula with reversal over $\Sigma_3$ in which all three variables are two-way is e-divisible by some element of $\Phi_3$, it suffices to check that every avoidable formula with reversal over $\Sigma_3$ whose fragments are all unavoidable and all have length at most $2$ is e-divisible by some element of $\Phi_3$.  This check is completed by modifying the algorithm in Figure \ref{Algorithm1} slightly, and applying it repeatedly until we have exhausted the irredundant unavoidable formulas over $\Sigma_3$ with fragments of length at most $2$ (there are none having $9$ or more fragments).
\item We show that every avoidable formula with reversal over $\Sigma_3$ in which exactly two variables are two-way is e-divisible by some avoidable formula with reversal whose fragments have length at most $4$, and which has at most two two-way variables; this is Lemma \ref{TwoTwoWay}.  So it suffices to check that every avoidable formula with reversal on variables $\{x,y,z,\rev{x},\rev{y}\}$, whose fragments are all unavoidable and have length at most $4$, is divisible by some element of $\Phi_3$.  Again, we use a modification of the algorithm given in Figure \ref{Algorithm1} to complete the check; we find there are no unavoidable formulas over $\Sigma_3$ satisfying the given conditions and having $16$ or more fragments.
\item Let $\phi$ be an avoidable formula with reversal over $\Sigma_3$ with exactly one two-way variable.  Without loss of generality, assume that $x$ is two-way in $\phi.$  We show that if $\phi$ has factors that flatten to $yxz$ and $zxy$, then $\phi$ is e-divisible by some element of $\Phi_3$; this is Lemma \ref{OneTwoWay}.  So this time, when we employ a modification of the algorithm given in Figure 1, we consider unavoidable fragments on variables $\{x,y,z,\rev{x}\}$, and we may exclude any fragments containing a factor that flattens to $zxy$ (without loss of generality).  There are no unavoidable formulas over $\Sigma_3$ having $10$ or more fragments satisfying these conditions.
\item We know that every avoidable formula with reversal over $\Sigma_3$ with no two-way variables (or equivalently, every classical formula on three variables) is divisible by some member of $\Phi_3$ since $\Phi_3$ contains a $3$-avoidance basis for classical formulas; c.f.\ \cite{ClarkThesis}.  \qedhere
\end{itemize}
\end{proof}

\begin{lemma}\label{ThreeTwoWay}
Let $\phi$ be a formula with reversal over $\Sigma_3$ in which all three variables are two-way.  If $\phi$ is avoidable and has a factor of length $3$, then $\phi$ is e-divisible by some avoidable formula with reversal over $\Sigma_3$ whose fragments all have length at most $2$.
\end{lemma}

\begin{proof}
Let $\phi$ be as in the lemma statement, and suppose that $\phi$ is avoidable and has a factor $p$ of length $3$.  We may assume that $\phi^\flat$ contains no squares, as otherwise $\phi$ is e-divisible by either $xx$ or $x\rev{x}$.  So up to relabelling, $p^\flat=xyx$ or $p^\flat=xyz.$  We show that $\restle{\phi}{2}$ is avoidable in each case.

\begin{description}

\item[Case I: $p^\flat=xyx$.] 

We claim that $\restle{\phi}{2}$, which e-divides $\phi$ through the inclusion map, is avoidable.  Note that $\restle{\phi}{2}^\flat$ contains $xy$ and $yx$.  Since both $x$ and $y$ are two-way in $\phi$ (and hence also in $\restle{\phi}{2}$), we see that $\restle{\phi}{2}$ is avoided by $(012)^\omega.$

\item[Case II: $p^\flat=xyz$.] 

If $\rest{\phi}{2}^\flat=\{xy,yz\}$, then $\phi$ divides the unavoidable formula $x^\sharp y^\sharp z^\sharp$, meaning that $\phi$ is unavoidable, contradicting our assumption.  So we may assume that $\rest{\phi}{2}^\flat$ contains at least one element from $\{xz,yx,zx,zy\}.$  This gives us four subcases.
\begin{description}
\item [Subcase a: $xz\in\rest{\phi}{2}^\flat$.] 

Since $\{xy,xz,yz\}\subseteq\rest{\phi}{2}^\flat$, we see that $\restle{\phi}{2}$ is avoided by $(012)^\omega$.

\item [Subcase b: $yx\in\rest{\phi}{2}^\flat$.] 

Since $\{xy,yx\}\subseteq\rest{\phi}{2}^\flat$, we see that $\restle{\phi}{2}$ is avoided by $(012)^\omega$.

\item [Subcase c: $zx\in\rest{\phi}{2}^\flat$.] 

Since $\{xy,yz,zx\}\subseteq\rest{\phi}{2}^\flat$, we see that $\restle{\phi}{2}$ is avoided by $(0123)^\omega$.

\item [Subcase d: $zy\in\rest{\phi}{2}^\flat$.] 

Since $\{yz,zy\}\subseteq\rest{\phi}{2}^\flat$, we see that $\restle{\phi}{2}$ is avoided by $(012)^\omega$.  \qedhere
\end{description}
\end{description}
\end{proof}

\begin{lemma}\label{TwoTwoWay}
Let $\phi$ be a formula with reversal over $\Sigma_3$ in which exactly two variables are two-way.  If $\phi$ is avoidable and has a factor of length $5$, then $\phi$ is e-divisible by some avoidable formula with reversal over $\Sigma_3$ whose fragments all have length at most $4$.
\end{lemma}

\begin{proof}
Let $\phi$ be as in the lemma statement, and suppose that $\phi$ is avoidable and has a factor $p$ of length $5.$  Without loss of generality, let $x$ be the unique one-way variable in $\phi.$  We may assume that neither $yy$ nor $zz$ is a factor of $p^\flat$, as otherwise $\phi$ is properly e-divisible by either $xx$ or $x\rev{x}.$  Further, we may assume that neither $yzy$ nor $zyz$ are factors of $p^\flat$, as otherwise $\restle{\phi}{2}$ is avoided by $(012)^\omega$.

If $x$ appears twice in $p$, then either $xx$, $xyx$, or $xyzx$ is a factor of $p$ (up to relabelling $y$ and $z$).  But then $\phi$ is properly e-divisible by avoidable formula with reversal $xx$, $xyx\cdot \rev{y},$ or $xyzx\cdot \rev{y}\cdot \rev{z},$ respectively.

So we may assume that $x$ appears at most once in $p$.  But since $yy$, $zz,$ $yzy,$ and $zyz$ are not factors of $p^\flat$, we see that $x$ must appear precisely in the middle of $p.$  Without loss of generality, we have two cases.

\begin{description}

\item[Case I: $p^\flat=yzxzy$.] Since $\rest{\phi}{2}^\flat$ contains $\{yz,zy\}$, $\restle{\phi}{2}$ is avoided by $(012)^\omega$, and hence $\phi$ is not minimal.

\item[Case II: $p^\flat=yzxyz$.] 

If $\rest{\phi}{2}^\flat=\{xy,yz,zx\},$ then since $x$ can appear at most once in any fragment of $\phi$ by the argument above, every fragment of $\phi$ flattens to some factor of $yzxyz$.  However, then $\phi$ divides the unavoidable formula $x^\sharp y^\sharp z x^\sharp y^\sharp$, and thus $\phi$ is unavoidable.  So we may assume that $\rest{\phi}{2}^\flat$ contains some factor from $\{xz,yx,zy\}$.
\begin{description}
\item [Subcase a: $xz\in\rest{\phi}{2}^\flat$.] 

Since $\{xz,xyz\}\subseteq\restle{\phi}{3}^\flat$, we see that $\restle{\phi}{3}$ is avoided by $(012)^\omega$.

\item [Subcase b: $yx\in\rest{\phi}{2}^\flat$.] 

Since $\{yx,yzx\}\subseteq\restle{\phi}{3}^\flat$, we see that $\restle{\phi}{3}$ is avoided by $(012)^\omega$.

\item [Subcase c: $zy\in\rest{\phi}{2}^\flat$.] 

Since $\{yz,zy\}\subseteq\restle{\phi}{2}^\flat$, we see that $\restle{\phi}{2}$ is avoided by $(012)^\omega$.  \qedhere
\end{description}

\end{description}  
\end{proof}

\begin{lemma}\label{OneTwoWay}
Let $\phi$ be an avoidable formula with reversal over $\Sigma_3$ in which $x$ is two-way while $y$ and $z$ are one-way.  If $yxz$ and $zxy$ are both factors of $\phi^\flat$, then $\phi$ is e-divisible by some element of $\Phi_3$.
\end{lemma}

\begin{proof}
Let $\phi$ be as in the lemma statement, and suppose that $p_1$ and $p_2$ are factors of $\phi$ that flatten to $yxz$ and $zxy$, respectively.  Without loss of generality, assume that $p_1=yxz$.  First off, if $\phi_2^\flat$ contains a square, then $\phi$ is e-divisible by $xx$ or $x\rev{x}$.  If $\phi$ contains the factor $yz,$ then the formula $yxz\cdot yz\cdot \rev{x}$ e-divides $\phi$ through the inclusion map, and we note that $xy\cdot xz\cdot yz\cdot \rev{y}\in \Phi_3$ e-divides $yxz\cdot yz\cdot \rev{x}$.  A similar argument applies if $\phi$ contains the factor $zy.$  If a factor of $\phi$ flattens to $yxy,$ then $\phi$ is e-divisible by $xyx\cdot \rev{y}\in\Phi_3$.  The situation is similar if a factor of $\phi$ flattens to $zxz.$

From the previous paragraph, we may now assume that $\phi^\flat$ does not contain any factors from $\{xx,yy,zz,yz,zy,yxy,zxz\}.$  Note that in particular, the only letter that can precede or follow $y$ or $z$ in a factor of $\phi^\flat$ is $x$.  Moreover, the only letter that can follow $yx$ or precede $xy$ is $z$, and the only letter that can follow $zx$ or precede $xz$ is $y$.

If $z$ appears at most once in each fragment of $\phi$, then by the observations of the previous paragraph, we see that every factor of $\phi$ flattens to a factor of $xyxzxyx$. But then $\phi$ is unavoidable, as it divides $x^\sharp yx^\sharp zx^\sharp yx^\sharp$ through the inclusion map. So we may assume that $z$ appears twice in some fragment of $\phi$.  By a similar argument, we may assume that $y$ appears twice in some fragment of $\phi$.  But then $\phi$ must have factors $p_y$ and $p_z$ that flatten to $yxzxy$ and $zxyxz,$ respectively.  We verify that $p_y\cdot p_z$ is e-equivalent to some member of $\Phi_3$ for any particular instance of $p_y$ and $p_z$, and hence $\phi$ is e-divisible by the corresponding element of $\Phi_3$.
\end{proof}

Now that we know $\Phi_3$ is a $3$-avoidance basis for formulas with reversal, the remainder of the article is devoted to demonstrating that $xyzx\cdot \rev{y}\cdot\rev{z}$ is the unique element in $\Phi_3$ of highest avoidability index $5$.  Many of the elements in $\Phi_3$ were already known to be $4$-avoidable, and several others are easily proven to be avoided by some infinite periodic word on at most $4$ letters; see Table~\ref{Avoidance3}.  We demonstrate that the remaining formulas with reversal in $\Phi_3$ are $4$-avoidable in the next two sections.

\section{The $3$-minimal formulas that flatten to $xyzyx\cdot zyxyz$ are $4$-avoidable}\label{xyzyxSection}

The formula $xyzyx\cdot zyxyz$ was proven to be $2$-avoidable in \cite{CircularFormulas}, where it was shown that it does not occur in the image of any $\left(\tfrac{5}{4}^+\right)$-free word over $5$ letters under a particular $15$-uniform morphism $m_{15}$; see \cite{CircularFormulas} for details.  Unfortunately, the other members of $\Phi_3$ that flatten to $xyzyx\cdot zyxyz$ all occur in the $m_{15}$-image of some $\left(\tfrac{5}{4}^+\right)$-free word (some prefix of the infinite $\left(\tfrac{5}{4}^+\right)$-free word described by Moulin-Ollagnier~\cite{Moulin1992}, in fact).  While we still suspect that these formulas have avoidability index $2$, we demonstrate in this section that they are at least $4$-avoidable.

For the remainder of this article, let $g=01/2/031/3$.  This morphism was used in \cite{ClarkThesis}, where it was shown that $g^\omega(0)$ avoids the $3$-minimal formulas $xyzyx\cdot zyxyz$ and $xyzx\cdot yzxy\cdot zyz$.  Here, we show that $g^\omega(0)$ avoids the following formulas:
\begin{align*}
xyzyx & \cdot z\rev{y}x\rev{y}z\\
xyz\rev{y}x & \cdot zyx\rev{y}z\\
xyzyx & \cdot zyx\rev{y}z\\
xyz\rev{y}x & \cdot z\rev{y}xyz
\end{align*}
It follows that the reversal of the third formula listed above is $4$-avoidable as well.  In fact, it appears to be avoided by $g^\omega(0)$ as well, although we do not prove this fact.  While the formulas listed above all have similar structure in that they flatten to $xyzyx\cdot zyxyz,$ we have not found a unified proof that they are avoided by $g^\omega(0)$; we treat the first two formulas together, but have an individual proof for each of the remaining formulas.

We use many Lemmas from Section 2.5 of \cite{ClarkThesis} on the structure of $g^\omega(0),$ which are stated below for ease of reference.  However, we encourage the reader to familiarize themselves with the relevant material in \cite{ClarkThesis} before continuing.  For a morphism $f\colon A^*\rightarrow B^*$, we use the symbol $\vert$ to denote \emph{cuts} in a word of the form $f(W),$ where $W\in A^*$.  The \emph{blocks} of $f$ are the words $f(a)$ for all letters $a\in A$, and a cut indicates the end of one block and the beginning of another (see \cite{ClarkThesis} for a more formal definition).  For example, in the word $g^\omega(0)$, the factor $301$ must appear as $\vert 3\vert 01\vert.$  We also use vertical bars to indicate the length of words, but the meaning will be clear from context.

\begin{lemma}[Clark, \cite{ClarkThesis}] \ 
\label{Clark}
\begin{enumerate}
\item \label{Suffix} If $XS$ is a factor of $g^\omega(0)$ with $|X|\geq 4$ and $S\in\{1,31\}$, then every appearance of $X$ in $g^\omega(0)$ is followed by $S$.
\item \label{Prefix} If $PX$ is a factor of $g^\omega(0)$ with $|X|\geq 4$ and $P\in\{0,03\}$, then every appearance of $X$ in $g^\omega(0)$ is preceded by $P$.
\item \label{xx} $g^\omega(0)$ avoids the pattern $xx$.
\item \label{xyxdotyxy} $g^\omega(0)$ avoids the formula $xyx\cdot yxy$.
\item \label{X3X} $g^\omega(0)$ avoids the pattern with constants $x3x$.
\item \label{Ends} Suppose $X\vert Y\vert X$ is a factor of $g^\omega(0)$ for some words $X$ and $Y$.  Then $X\vert Y\vert X=\vert X\vert Y\vert X\vert.$
\item \label{Last} For all words $X,Y\in\{0,1,2,3\}^*$ and all letters $z\in\{0,1,2\},$ $g^\omega(0)$ does not contain both $XzY3Xz$ and $zY3XzY$.
\end{enumerate}
\end{lemma}

In many of the proofs in this section, we perform an exhaustive search of all small factors of $g^\omega(0)$ to eliminate certain possibilities.  For a fixed $\ell$, the set $S_\ell$ of all factors of $g^\omega(0)$ of length $\ell$ can be enumerated by the following algorithm:
\begin{enumerate}[label=\arabic*.]
\item Find the length $\ell$ prefix $P_\ell$ of $g^\omega(0)$.
\item Set $S_\ell=\{P_\ell\}$.
\item While the set $N=\{U\colon\ U \mbox{ is a factor of } g(V) \mbox{ for some } V\in S_\ell \mbox{ and } |U|=\ell\}$ is not contained in $S_\ell$, set $S_\ell=N$.
\end{enumerate}
The correctness of the algorithm is easily verified.  Every length $\ell$ factor of $g^\omega(0)$ must arise as a factor of the $g$-image of another length $\ell$ factor since $g$ is nonerasing, and the algorithm must terminate because $S_\ell$ is finite.  Once the set $S_\ell$ is obtained, it is straightforward to find the smallest value $n_\ell$ such that every factor of length $\ell$ appears in $g^{n_\ell}(0)$.  We completed these computations for small values of $\ell,$ and the results are summarized in Table \ref{IterationTable}.

\begin{table}[h]
\begin{center}
\begin{tabular}{c c}
$\ell$ & $n_\ell$\\\midrule
$1$ & $4$\\
$2$ & $6$\\
$[3,4]$ & $8$\\
$5$ & $9$\\
$[6,8]$ & $10$\\
$[9,12]$ & $11$\\
$[13,19]$ & $12$\\
$[20,30]$ & $13$\\
$[31,48]$ & $14$
\end{tabular}
\caption{The smallest value $n_\ell$ such that all factors of $g^\omega(0)$ of length $\ell$ appear in $g^{n_\ell}(0).$  Here $[a,b]$ denotes any integer between $a$ and $b$ inclusive.}
\label{IterationTable}
\end{center}
\end{table}

We first use an exhaustive search to determine the \emph{reversible factors} of $g^\omega(0)$.  A word $u$ is called a \emph{reversible factor} of a word $w$ if both $u$ and its reversal $\revant{u}$ are factors of $w$.

\begin{lemma}\label{ReversibleFactors}
The reversible factors of $g^\omega(0)$ are exactly the factors of $303$, $323$, $03130$, and $31013$.  The only nonpalindromic reversible factors are $30,$ $32,$ $31,$ $031,$ $0313,$ $10,$ $310,$ $3101,$ and their reversals.
\end{lemma}

\begin{proof}
Reading from Table \ref{IterationTable}, all factors of $g^\omega(0)$ of length $6$ appear in $g^{10}(0)$.  By an exhaustive search of $g^{10}(0)$, we conclude that $g^\omega(0)$ has no reversible factors of length $6$ or more, and that the reversible factors of $g^\omega(0)$ are exactly those indicated in the first statement.  The second statement follows immediately.
\end{proof}

Note that since $g^\omega(0)$ has only finitely many reversible factors, we could in theory use Cassaigne's algorithm~\cite{CassaigneThesis} to show that $g^\omega(0)$ avoids the formulas that flatten to $xyzyx\cdot zyxyz$.  One would need to carry out Cassaigne's algorithm on a finite set of formulas with constants (one for each reversible factor).  However, since the proofs are feasible by hand, we write the proofs instead of implementing Cassaigne's algorithm.  The written proofs provide some nice insight into the structure of $g^\omega(0)$.

We will require the following straightforward corollary to Lemma~\ref{Clark} parts \ref{Suffix} and \ref{Prefix}.

\begin{corollary}
\label{Preimage}
Let $X$ be a word of length at least $9$.
\begin{enumerate}
\item \label{Suffix2} If $XS$ is a factor of $g^\omega(0)$ for some $S\in\{2,32\}$, then every appearance of $X$ in $g^\omega(0)$ is followed by $S$.  
\item \label{Prefix2} If $PX$ is a factor of $g^\omega(0)$ for some $P\in\{01,013\},$ then every appearance of $X$ in $g^\omega(0)$ is preceded by $P$.  
\end{enumerate}
\end{corollary}

\begin{proof}
For part \ref{Suffix2}, suppose that $XS$ is a factor of $g^\omega(0)$ with $|X|\geq 9$ and $S\in\{2,32\}$.  Note that we have $X|S$, and let $X_1$ denote the length $4$ suffix of any preimage of $X$ in $g$.  Note that $X_1$ is completely determined since $|X|\geq 9$, the image of every factor of length $4$ of $g^\omega(0)$ has length at most $9$ (verified directly by computer), and $g$ is injective.  Let $S_1$ denote the preimage of $S$ in $g$, so $S_1\in\{1,31\}.$

Now $X_1S_1$ is a factor of $g^\omega(0)$ as well, and by Lemma~\ref{Clark}\ref{Suffix}, every appearance of $X_1$ in $g^\omega(0)$ is followed by $S_1$.  So every appearance of the preimage of $X$ is followed by $S_1$, meaning that every appearance of $X$ is followed by $g(S_1)=S,$ as desired.

The proof of part~\ref{Prefix2} is analogous.
\end{proof}

We are now ready to prove the first main result of this section.

\begin{theorem}\label{gomega1}
The following formulas with reversal are avoided by $g^\omega(0)$:
\begin{center}
\begin{tabular}{l}
$xyzyx\cdot z\rev{y}x\rev{y}z$ \\
$xyz\rev{y}x\cdot zyx\rev{y}z$
\end{tabular}
\end{center}
\end{theorem}

\begin{proof}
Let $\phi$ be any formula from the list in the theorem statement, and suppose towards a contradiction that $\phi$ occurs in $g^\omega(0)$ through morphism respecting reversal $f$.  Let $Y=f(y)$, and hence $\rev{Y}=f\left(\rev{y}\right)$, as $f$ respects reversal.  Since $\phi^\flat=xyzyx\cdot zyxyz$ is avoided by $g^\omega(0)$, we must have $Y\neq \rev{Y}$.  In other words, $Y$ is a nonpalindromic reversible factor of $g^\omega(0)$.  By Lemma \ref{ReversibleFactors}, $Y$ must be one of $30,$ $32,$ $31,$ $031,$ $0313,$ $10,$ $310,$ or $3101,$ or the reversal of one of these.  All of our arguments here are symmetric in $Y$ and $\rev{Y}$, so we assume without loss of generality that $Y\in \{03,32,31,130,3130,10,310,3101\}.$

Before we begin with case work, we note that if $|f(x)|\leq 8$ and $|f(z)|\leq 8$, then each fragment of $\phi$ has $f$-image of length at most $32$ (since $|Y|\leq 4$).  From Table \ref{IterationTable}, every factor of length at most $32$ of $g^\omega(0)$ appears in $g^{14}(0).$  We verify by exhaustive search that $\phi$ does not occur in $g^{14}(0)$, and hence we may now assume that $|f(a)|\geq 9$ for some variable $a$ in $\{x,z\}$.  Let $A=f(a)$.  Note that $\phi$ has the factors $ay,$ $a\rev{y},$ $ya,$ and $\rev{y}a$ (independent of whether $a=x$ or $a=z$), and hence $g^\omega(0)$ has factors $AY$, $A\rev{Y}$, $YA,$ and $\rev{Y}A$.  The assumption that $|A|\geq 9$ is used throughout.

\begin{case}
\item[$Y\in \{10,310,130,3130\}$] We have $YA=Y'\vert 0A$ for some $Y'\in\{1,31,13,313\}$, from which we conclude that $A$ must have prefix $1$ or $31.$  However, then $\rev{Y}A$ is not a factor of $g^\omega(0)$, a contradiction.  For example, if $Y=10$, then $\rev{Y}A=01A$, and $11$ and $131$ are not factors of $g^\omega(0).$  The other possibilities for $Y$ are eliminated similarly.

\item[$Y=3101$] We have $AY=A31\vert 01\vert$, meaning that $A$ must end in $0$.  However, then $A\rev{Y}=A1\vert 01 \vert 3$ is not a factor of $g^\omega(0)$, because $0101$ is not a factor of $g^\omega(0)$ by Lemma \ref{Clark}\ref{xx}.

\item[$Y=31$] Since the factor $AY=A31$ appears in $g^\omega(0)$, by Lemma \ref{Clark}\ref{Suffix}, every appearance of $A$ in $g^\omega(0)$ is followed by $31$.  But then $A\rev{Y}=A13$ cannot be a factor of $g^\omega(0)$, a contradiction.

\item[$Y=03$] Since the factor $YA=03A$ appears in $g^\omega(0)$, by Lemma \ref{Clark}\ref{Prefix}, every appearance of $A$ in $g^\omega(0)$ is preceded by $03$.  But then $\rev{Y}A=30A$ cannot be a factor of $g^\omega(0)$, a contradiction.

\item[$Y=32$] Since the factor $AY=A32$ appears in $g^\omega(0)$, by Corollary~\ref{Preimage}\ref{Suffix}, every appearance of $A$ in $g^\omega(0)$ is followed by $32$.  But then $A\rev{Y}=A23$ cannot be a factor of $g^\omega(0)$, a contradiction. \qedhere
\end{case}
\end{proof}

Next, we demonstrate that $xyzyx\cdot zyx\rev{y}z$ is avoided by $g^\omega(0)$.  It follows that the reversed formula $xyzyx\cdot z\rev{y}xyz$ is $4$-avoidable as well.  While many ideas from the proof of Theorem \ref{gomega1} are used in the proof that $xyzyx\cdot zyx\rev{y}z$ is avoided by $g^\omega(0)$, minor adjustments are required because the factors $z\rev{y}$ and $\rev{y}x$ are not present in $xyzyx\cdot zyx\rev{y}z$.

\begin{theorem}\label{gomega2}
The formula with reversal $xyzyx\cdot zyx\rev{y}z$ is avoided by $g^\omega(0)$.  
\end{theorem}

\begin{proof}
Suppose that $\phi=xyzyx\cdot zyx\rev{y}z$ occurs in $g^\omega(0)$ through morphism respecting reversal $f$.  Let $X=f(x),$ $Y=f(y),$ and $Z=f(z)$.  By the same reasoning as in the proof of Theorem~\ref{gomega1}, $Y$ must be one of $30,$ $32,$ $31,$ $031,$ $0313,$ $10,$ $310,$ or $3101,$ or the reversal of one of these.  We may also assume that either $|X|\geq 9$ or $|Z|\geq 9$ by verifying that $\phi$ does not occur in $g^{14}(0)$.  In the case work completed below, we frequently use either the fact that $g^\omega(0)$ has factors $YZ$ and $\rev{Y}Z$, or the fact that $g^\omega(0)$ has factors $XY$ and $X\rev{Y}$.  

\begin{case}
\item[$Y\in\{10,310,130,3130\}$] By the argument used for Case 1 in the proof of Theorem \ref{gomega1} with $A$ replaced by $Z$, we are done if $Z$ has prefix $1$ or $31$.  The only remaining possibility is that $Z=3$ (as we do not necessarily have $|Z|>1$ here).  But in that case $YZY=Y3Y$ is not a factor of $g^\omega(0)$ by Lemma~\ref{Clark}\ref{X3X}, a contradiction.  A similar argument applies if $Y$ is the reversal of any word in $\{10,310,130,3130\}$ instead.

\item[$Y=3101$] The argument used in Case 2 in the proof of Theorem \ref{gomega1}  applies here, with $A$ replaced by $X.$  A symmetric argument works if $Y=1013.$

\item[$Y=31$] Since $XY=X31,$ we conclude that $X$ ends in $0$.  Then consider the factor $X\rev{Y}Z=X13Z=X'013Z.$  Taking the preimage of $013$ in $g$, we obtain $03$, which must be followed by $1$.  So $Z$ begins with $g(1)=2$.  But then $YZ$ has prefix $312,$ and this factor never appears in $g^\omega(0)$ (verified computationally).  By swapping $Y$ and $\rev{Y}$ in this argument, we rule out the case $Y=13$.

\item[$Y=03$] If $|Z|\geq 4$, then by Lemma \ref{Clark}\ref{Prefix}, since $YZ=03Z$ appears in $g^\omega(0)$, every appearance of $Z$ is preceded by $03$.  But then $\rev{Y}Z=30Z$ does not appear in $g^\omega(0)$, a contradiction.  So we may now assume that $|Z|\leq 3$.  From the factor $YZ=03Z,$ we see that $Z=1Z'$ for some word $Z'$ with $|Z'|\leq 2$.  Now consider the factor $YZY=\vert 031\vert Z'\vert 03,$ which has preimage $2W2$ in $g$, where $g(W)=Z'$.  Since $|Z'|\leq 2,$ we must have $W\in\{\varepsilon,0,1,3,13,31\},$ and we verify computationally that $2W2$ is not a factor of $g^\omega(0)$ in each of these cases.

\item[$Y=30$] By an argument symmetric to the one used in the previous case, we may assume that $|Z|\leq 3$.  Since $\rev{Y}Z=03Z$ appears in $g^\omega(0)$, $Z$ begins with $1$, and since $ZY=Z30$ appears in $g^\omega(0),$ $Z$ ends with either $1$ or $2$.  By examining all small factors of $g^\omega(0)$, we find that $Z\in\{1,12,101,132\}.$  But in each case, either $YZY$ or $\rev{Y}Z$ is not a factor of $g^\omega(0)$ (verified computationally), a contradiction.

\item[$Y=32$] If $|X|\geq 9$, then the argument of Case 5 in the proof of Theorem \ref{gomega1} applies with $A$ replaced by $X$, so we may assume that $|X|\leq 8$ and $|Z|\geq 9$.  Consider the factor $YX\rev{Y}=32X23.$  By considering the list of factors of $g^\omega(0)$ of length at most $12$, the only possibilities for $X$ are $01$ and $301$.  We show that the factors $XYZ$ and $X\rev{Y}Z$ cannot both appear in $g^\omega(0)$ in either case, a contradiction.  Whether $X=01$ or $X=301$, $XYZ$ and $X\rev{Y}Z$ have suffixes
\[
0132Z \mbox{ and } 0123Z.
\]
If these factors lie in $g^n(0)$, then taking preimages in $g$ twice, the factors $2a$ and $03a$ lie in $g^{n-2}(0)$ for some letter $a$, where $g^2(a)$ is a prefix of $Z$ ($Z$ is long enough to determine $a$ completely).  The factor $03a$ forces $a=1$, but the factor $21$ never appears in $g^\omega(0)$.

\item[$Y=23$] By similar reasoning to the previous case, we may assume that $|X|\leq 8$.  Consider the factor $YX\rev{Y}=23x32.$  By considering the list of factors of $g^\omega(0)$ of length at most $12$, we find that $X=01203101$ is the only possibility.  However, since $X$ ends in $01$, the argument used in the previous case demonstrates once again that the factors $XYZ$ and $X\rev{Y}Z$ cannot both appear in $g^\omega(0)$, a contradiction.\qedhere
\end{case}
\end{proof}

It is slightly trickier to show that the formula $xyz\rev{y}x\cdot z\rev{y}xyz$ is avoided by $g^\omega(0)$.  In this formula, every appearance of $y$ is preceded by $x$ and followed by $z$, and every appearance of $\rev{y}$ is preceded by $z$ and followed by $x$, meaning that very few of the arguments used in the proofs of Theorem \ref{gomega1} and Theorem \ref{gomega2} apply.  We first prove two useful lemmas.

\begin{lemma}\label{xz3xdotz3xz}
The word $g^\omega(0)$ avoids the formula with constants $xz3x\cdot z3xz.$
\end{lemma}

\begin{proof}
Suppose towards a contradiction that $xz3x\cdot z3xz$ occurs in $g^m(0)$, and that $m$ is minimal with respect to this property.  Moverover, let $xz3x\cdot z3xz$ occur in $g^m(0)$ through nonerasing morphism $f,$ and let $X=f(x)$ and $Z=f(z)$.  Note that if $XZ$ has length at most $4$ then $XZ3X$ and $Z3XZ$ each have length at most $8$.  The prefix $g^{10}(0)$ contains all factors of $g^\omega(0)$ of length at most $8$ (see Table~\ref{IterationTable}).  We verify by exhaustive search that $g^{10}(0)$ avoids $xz3x\cdot z3xz,$ and we may now assume that $XZ$ has length at least $5$.

First we show that the factors $XZ3X$ and $Z3XZ$ must parse as $XZ\vert 3\vert X$ and $Z\vert 3\vert XZ.$  Otherwise, we would have $Z=Z'0$ and $X=1X'$, giving $XZ3X=1X'Z'031X'$ and $Z3XZ=Z'031X'Z'0.$  By Lemma~\ref{Clark}\ref{Suffix}, every appearance of $XZ=1X'Z'0$ must be followed by $31$ (since $1X'Z'$ has length at least $4$ by the assumption that $|XZ|\geq 5$), so $Z'\vert 031\vert X'Z'\vert 0$ appears internally as $Z'031X'Z'031.$  However, then $xyx\cdot yxy$ occurs in $g^m(0)$ as indicated below:
\[
\underbrace{X'}_{\overset{\upmapsto}{x}}
\underbrace{Z'031}_{\overset{\upmapsto}{y}}
\underbrace{X'}_{\overset{\upmapsto}{x}}
\mbox{ and }
\underbrace{Z'031}_{\overset{\upmapsto}{y}}
\underbrace{X'}_{\overset{\upmapsto}{x}}
\underbrace{Z'031}_{\overset{\upmapsto}{y}}.
\]

We may now assume that we have factors $XZ\vert 3\vert X$ and $Z\vert 3\vert XZ$.  We show next that we must have $X\vert Z$.  If this is not the case then we have $XZ=X'0L1Z',$ where $L\in\{\varepsilon,3\}.$ Moreover, $X'0$ is a prefix of $X$ and $1Z'$ is a suffix of $Z$.  So $XZ3X$ and $Z3XZ$ contain the factors
\[
X'0L1Z'3X'0 
\mbox{ and }
1Z'3X'0L1Z'.
\]
By Lemma~\ref{Clark}\ref{Prefix}, every appearance of $1Z'3X'0$ is preceded by $0L,$ meaning that the latter factor appears internally as $0L1Z'3X'0L1Z'.$  However, then Lemma~\ref{Clark}\ref{Last} is contradicted as follows:
\[
\underbrace{X'}_{\overset{\upmapsto}{X}}
\underbrace{0}_{\overset{\upmapsto}{z}}
\underbrace{L1Z'}_{\overset{\upmapsto}{Y}}
3
\underbrace{X'}_{\overset{\upmapsto}{X}}
\underbrace{0}_{\overset{\upmapsto}{z}}
\mbox{ and }
\underbrace{0}_{\overset{\upmapsto}{z}}
\underbrace{L1Z'}_{\overset{\upmapsto}{Y}}
3
\underbrace{X'}_{\overset{\upmapsto}{X}}
\underbrace{0}_{\overset{\upmapsto}{z}}
\underbrace{L1Z'}_{\overset{\upmapsto}{Y}}.
\]

So we have factors $X\vert Z\vert 3\vert X$ and $Z\vert 3\vert X\vert Z$.  By Lemma~\ref{Clark}\ref{Ends}, we actually have factors $\vert X\vert Z\vert 3\vert X\vert$ and $\vert Z\vert 3\vert X\vert Z\vert$.  However, the preimages of these factors appear in $g^{m-1}(0)$, and contradict the minimality of $m$.
\end{proof}

\begin{lemma}\label{x2z3xdotz3x2z}
The word $g^\omega(0)$ avoids the formula with constants $x2z3x\cdot z3x2z.$
\end{lemma}

\begin{proof}
Suppose towards a contradiction that $x2z3x\cdot z3x2z$ occurs in $g^\omega(0)$ through morphism $f$, and let $X=f(x)$ and $Z=f(z)$.  Then the factors $X2Z3X$ and $Z3X2Z$ appear in $g^\omega(0)$.  We may assume that $|Z3X|\geq 9$, since otherwise these factors each have length at most $15$.  All factors of $g^\omega(0)$ of length at most $15$ appear in $g^{12}(0),$ and we have verified by exhaustive search that the given formula does not occur in $g^{12}(0).$

Assuming now that $|Z3X|\geq 9,$ by Corollary~\ref{Preimage}\ref{Suffix2}, every appearance of $Z3X$ is followed by $2$, meaning that $X2Z3X$ and $Z3X2Z$ appear internally as
\[
\underbrace{X2}_{\overset{\upmapsto}{x}}
\underbrace{Z}_{\overset{\upmapsto}{z}}
3
\underbrace{X2}_{\overset{\upmapsto}{x}}
\mbox{ and }
\underbrace{Z}_{\overset{\upmapsto}{z}}
3
\underbrace{X2}_{\overset{\upmapsto}{x}}
\underbrace{Z}_{\overset{\upmapsto}{z}},
\]
which contradicts Lemma~\ref{xz3xdotz3xz} as indicated.
\end{proof}

Now we are ready to show the last main result of this section.

\begin{theorem}\label{gomega3}
The formula with reversal $xyz\rev{y}x\cdot z\rev{y}xyz$ is avoided by $g^\omega(0)$.
\end{theorem}

\begin{proof}
Let $\phi=xyz\rev{y}x\cdot z\rev{y}xyz$, and suppose that $\phi$ occurs in $g^\omega(0)$ through morphism respecting reversal $f$.  Let $X=f(x),$ $Y=f(y),$ and $Z=f(z).$  By Lemma~\ref{ReversibleFactors} and the symmetry of $y$ and $\rev{y}$ in $\phi$ (by swapping $x$ and $z$), we may assume that $Y$ is in $\{03,32,31,031,3130,10,310,3101\}$.  We examine each case separately.  We may assume that $|X|+|Z|\geq 7,$ as otherwise the $f$-image of each fragment of $\phi$ has length at most $19$, and we have verified computationally that $g^{12}(0)$ avoids $\phi.$  Our general technique is to use Lemma~\ref{Clark} parts \ref{Suffix} and \ref{Prefix} and Corollary~\ref{Preimage} parts \ref{Suffix2} and \ref{Prefix2} to extend the factors $XYZ\rev{Y}X$ and $Z\rev{Y}XYZ$ to words that constitute an occurrence of a formula that we know is avoided by $g^\omega(0)$.  The assumption $|X|+|Z|\geq 7$, along with the fact that $|Y|\geq 2,$ ensures that $XYZ$ and $Z\rev{Y}X$ have length at least $9$, allowing each of the aforementioned lemmas to be applied.

\begin{case}
\item[$Y=10$]
In this case, the factors $X10Z01X$ and $Z01X10Z$ appear in $g^\omega(0).$  By Lemma~\ref{Clark}\ref{Suffix}, every appearance of $Z01X$ is followed by $1$, while by Lemma~\ref{Clark}\ref{Prefix}, every appearance of $Z01X$ is preceded by $0$. Finally, by Corollary~\ref{Preimage}\ref{Prefix2}, every appearance of $X10Z$ is followed by $01$.  Therefore, the factors appear internally as
\[
\underbrace{X1}_{\overset{\upmapsto}{x}}
\underbrace{0Z01}_{\overset{\upmapsto}{y}}
\underbrace{X1}_{\overset{\upmapsto}{x}}
\mbox{ and }
\underbrace{0Z01}_{\overset{\upmapsto}{y}}
\underbrace{X1}_{\overset{\upmapsto}{x}}
\underbrace{0Z01}_{\overset{\upmapsto}{y}},
\]
which contradicts Lemma~\ref{Clark}\ref{xyxdotyxy} as indicated.

\item[$Y=310$] 
In this case, the factors $X310Z013X$ and $Z013X310Z$ appear in $g^\omega(0)$.  By Lemma~\ref{Clark}\ref{Prefix}, every appearance of $Z013X$ is preceded by $0$, and by Lemma~\ref{Clark}\ref{Suffix}, every appearance of $Z013X$ is followed by $31$.  Finally, by Corollary~\ref{Preimage}\ref{Prefix2}, every appearance of $X310Z$ is preceded by $013$.  Therefore, the factors appear internally as
\[
\underbrace{X31}_{\overset{\upmapsto}{x}}
\underbrace{0Z013}_{\overset{\upmapsto}{y}}
\underbrace{X31}_{\overset{\upmapsto}{x}}
\mbox{ and }
\underbrace{0Z013}_{\overset{\upmapsto}{y}}
\underbrace{X31}_{\overset{\upmapsto}{x}}
\underbrace{0Z013}_{\overset{\upmapsto}{y}},
\]
which contradicts Lemma~\ref{Clark}\ref{xyxdotyxy} as indicated.

\item[$Y=3130$]
In this case, the factors $X3130Z0313X$ and $Z0313X3130Z$ appear in $g^\omega(0)$.  Note that $X=X'0$ for some word $X'$ since the factor $X31$ is present.  Note that $X'\neq \emptyset$ (if it was then $0313X313=03130313$ is a square, contradicting Lemma~\ref{Clark}\ref{xx}).  Now, substituting $X=X'0$ and applying Lemma~\ref{Clark}\ref{Prefix} to conclude that every appearance of $Z0313$ is preceded by $0$, the factors above appear internally as
\[
\underbrace{X'}_{\overset{\upmapsto}{x}}
\underbrace{031 3}_{\overset{\upmapsto}{y}} 
\underbrace{0Z}_{\overset{\upmapsto}{z}}
\underbrace{0313}_{\overset{\upmapsto}{y}} 
\underbrace{X'}_{\overset{\upmapsto}{x}} 0
\mbox{ and } 
\underbrace{0Z}_{\overset{\upmapsto}{z}} 
\underbrace{031 3}_{\overset{\upmapsto}{y}} 
\underbrace{X'}_{\overset{\upmapsto}{x}}
\underbrace{031 3}_{\overset{\upmapsto}{y}} 
\underbrace{0Z}_{\overset{\upmapsto}{z}}.
\]
Thus $xyzyx\cdot zyxyz$ occurs in $g^\omega(0)$ as indicated, a contradiction.

\item[$Y=3101$]
Using Lemma~\ref{Clark} parts \ref{Suffix} and \ref{Prefix}, and Corollary~\ref{Preimage} parts \ref{Suffix2} and \ref{Prefix2}, the factors $X3101Z1013X$ and $Z1013X 3101 Z$ appear internally as
\[
\underbrace{013X31}_{\overset{\upmapsto}{x}}
\underbrace{01Z1}_{\overset{\upmapsto}{y}}
\underbrace{013X31}_{\overset{\upmapsto}{x}}
\mbox{ and }
\underbrace{01Z1}_{\overset{\upmapsto}{y}}
\underbrace{013X31}_{\overset{\upmapsto}{x}}
\underbrace{01Z1}_{\overset{\upmapsto}{y}},
\]
which contradicts Lemma~\ref{Clark}\ref{xyxdotyxy} as indicated.

\item[$Y=32$] In this case, the factors $X32Z23X$ and $Z23X32Z$ appear in $g^\omega(0).$  By Corollary~\ref{Preimage}\ref{Suffix2}, every appearance of $X32Z$ is followed by $2$ and every appearance of $Z23X$ is followed by $32.$  Therefore, the factors appear internally as
\[
\underbrace{X32}_{\overset{\upmapsto}{x}}
\underbrace{Z2}_{\overset{\upmapsto}{z}}
3
\underbrace{X32}_{\overset{\upmapsto}{x}}
\mbox{ and }
\underbrace{Z2}_{\overset{\upmapsto}{z}}
3
\underbrace{X32}_{\overset{\upmapsto}{x}}
\underbrace{Z2}_{\overset{\upmapsto}{z}},
\]
which contradicts Lemma~\ref{xz3xdotz3xz} as indicated.

\item[$Y=03$] In this case, the factors $X03Z30X$ and $Z30X03Z$ appear in $g^\omega(0)$.  By Lemma~\ref{Clark}\ref{Prefix}, every appearance of $X03Z$ is preceded by $0$, and every appearance of $Z30X$ is preceded by $03$.  Thus the above factors appear internally as
\[
\underbrace{0X}_{\overset{\upmapsto}{x}}
\underbrace{03Z}_{\overset{\upmapsto}{z}}
3
\underbrace{0X}_{\overset{\upmapsto}{x}}
\mbox{ and }
\underbrace{03Z}_{\overset{\upmapsto}{z}}
3
\underbrace{0X}_{\overset{\upmapsto}{x}}
\underbrace{03Z}_{\overset{\upmapsto}{z}},
\] 
which contradicts Lemma~\ref{xz3xdotz3xz} as indicated.

\item[$Y=31$] In this case, the factors $X31Z13X$ and $Z13X31Z$ appear in $g^\omega(0)$.  By Lemma~\ref{Clark}\ref{Suffix}, every appearance of $Z13X$ is followed by $31$, while every appearance of $X31Z$ is followed by $1$.  Thus the above factors appear internally as
\[
\underbrace{X31}_{\overset{\upmapsto}{x}}
\underbrace{Z1}_{\overset{\upmapsto}{z}}
3
\underbrace{X31}_{\overset{\upmapsto}{x}}
\mbox{ and }
\underbrace{Z1}_{\overset{\upmapsto}{z}}
3
\underbrace{X31}_{\overset{\upmapsto}{x}}
\underbrace{Z1}_{\overset{\upmapsto}{z}},
\] 
which contradicts Lemma~\ref{xz3xdotz3xz} as indicated.

\item[$Y=031$] In this case, the factors $X031Z130X$ and $Z130X031Z$ appear in $g^\omega(0).$  By Lemma~\ref{Clark}\ref{Prefix}, every appearance of $X031Z$ is preceded by $0$, while by Lemma~\ref{Clark}\ref{Suffix}, every appearance of $Z130X$ is followed by $1$.  Finally, by Lemma~\ref{Clark}\ref{Ends}, the factors above appear internally as
\[
\vert 0X \vert 031 \vert Z1 \vert 3 \vert 0X \vert 
\mbox{ and }
\vert Z1 \vert 3 \vert 0X \vert 031 \vert Z1 \vert.
\]
Now let $X'$ denote the preimage of $0X,$ and $Z'$ denote the preimage of $Z1.$ The preimage of the factors above can then be written
\[
X'2Z'3X' \mbox{ and } Z'3X'2Z'.
\]
These preimages must appear in $g^\omega(0),$ which contradicts Lemma~\ref{x2z3xdotz3x2z}.  \qedhere
\end{case}
\end{proof}

\section{The formulas $xyzx\cdot yzxy\cdot \rev{z}$ and $xyzx\cdot y\rev{z}xy$ are $4$-avoidable}\label{xyzxSection}

For the remainder of the article, let $h=01/12/20/3$.  While $xyzx\cdot yzxy\cdot \rev{z}$ and $xyzx\cdot y\rev{z}xy$ are not avoided by $g^\omega(0)$, they are avoided by $h(g^\omega(0))$, as proven in this section.  This is interesting in light of the fact that for classical formulas, if a formula $\phi$ occurs in some word $w$, then it also occurs in $f(w)$ for any nonerasing morphism $f$.  Evidently, this is not the case for formulas with reversal.  We start with some straightforward lemmas.

\begin{lemma}\label{UniqueParsing}
Every factor of $h(g^\omega(0))$ of length at least $3$ parses uniquely into code words of $h$, i.e.~has unique preimage in $h$.
\end{lemma}

\begin{proof}
Let $u$ be a factor of $h(g^\omega(0))$ of length at least $3$.  We need only show the existence of a single cut in $u$, as each code word of $h$ begins with a different letter from the others, and ends with a different letter from the others.  Thus a single cut determines a parsing for the entire word.

If the letter $3$ appears in $u$ then we have cuts on either side of $3$.  If $aa$ appears for some $a\in\{0,1,2\}$ then we have $a\vert a$.  The factors $10,$ $21$, $02$, and $012$ do not appear in $h(g^\omega(0))$, so the only remaining possibilities for $u$ are $120$ and $201$.  However, $1\vert 20$ is impossible because preimage $02$ does not appear in $g^\omega(0)$, so we have $12\vert 0.$  Similarly, $20\vert 1$ is impossible because preimage $21$ is not a factor of $g^\omega(0)$, so we have $2\vert 01.$
\end{proof}

\begin{lemma}\label{Extend}
Let $X$ be a factor of $h(g^\omega(0))$ of length at least $3$.
\begin{enumerate}
\item \label{Suffix3} If $Xb\vert$ is a factor of $h(g^\omega(0))$ for some letter $b\in\{0,1,2\},$ then every appearance of $X$ in $h(g^\omega(0))$ is followed by $b$.
\item \label{Prefix3} If $\vert bX$ is a factor of $h(g^\omega(0))$ for some letter $b\in\{0,1,2\},$ then every appearance of $X$ in $h(g^\omega(0))$ is preceded by $b$.
\end{enumerate}
\end{lemma}

\begin{proof}
For \ref{Suffix3}, let $a=b-1,$ where the subtraction is modulo $3$.  We must have $Xb\vert=X'\vert ab\vert,$ and hence by Lemma \ref{UniqueParsing}, every appearance of $X$ parses as $X'\vert a$, meaning that every appearance of $X$ is followed by $b$.  The proof of \ref{Prefix3} is similar.
\end{proof}

\begin{lemma}\label{hAvoid3}
The word $h(g^\omega(0))$ avoids $xy3x\cdot y3xy$.
\end{lemma}

\begin{proof}
Suppose towards a contradiction that $xy3x\cdot y3xy$ occurs in $h(g^\omega(0))$ through nonerasing morphism $f$.  Let $X=f(x)$ and $Y=f(y)$.  By exhaustive search of $h(g^8(0)),$ which contains all factors of $h(g^\omega(0))$ of length at most $4,$ we may assume that either $X$ or $Y$ has length at least $2$.  Consider the factors
\[
XY3X \mbox{ and } Y3XY,
\]
which appear in $h(g^\omega(0))$ by supposition.  By Lemma~\ref{UniqueParsing}, the word $XY$ has a unique parsing, so in particular we have
\[
\vert XY\vert 3\vert X \mbox{ and } Y\vert 3\vert XY\vert.
\]
Now we either have $X\vert Y,$ or $X=X'a$, $Y=bY',$ and $XY=X\vert ab\vert Y$ for some $a,b\in\{0,1,2\}$ satisfying $a+1\equiv b\pmod{3}$.  In each case, we will show that the $h$-preimages of $XY3X$ and $Y3XY$ give an occurrence of $xy3x\cdot y3xy$ in $g^\omega(0)$, contradicting Lemma~\ref{xz3xdotz3xz}.
\begin{case}
\item[$XY=X\vert Y$] 
In this case,
\[
XY3X=\vert X\vert Y\vert 3\vert X\vert \mbox{ and } Y3XY=\vert Y\vert 3\vert X\vert Y\vert.
\]
The preimages of these factors in $h$ very clearly give an occurrence of $xy3x\cdot y3xy$ in $g^\omega(0)$, a contradiction.
\item[$X=X'a$, $Y=bY'$, and $XY=X\vert ab\vert Y$]
In this case,
\[
XY3X=X'\vert ab\vert Y'\vert 3\vert X'a
\mbox{ and }
Y3XY=bY'\vert 3\vert X'\vert ab\vert Y'. 
\]
Further, by Lemma~\ref{Extend}\ref{Suffix3} and \ref{Prefix3}, every appearance of $bY'3X'a$ is followed by $b$ and preceded by $a.$  So the above factors appear internally as
\[
\vert X'\vert ab\vert Y'\vert 3\vert X'\vert ab\vert
\mbox{ and }
\vert ab\vert Y'\vert 3\vert X'\vert ab\vert Y'\vert,
\]
where Lemma~\ref{UniqueParsing} was applied to determine some cuts.  Taking preimages in $h$, we obtain
\[
X_1aY_13X_1a \mbox{ and } aY_13X_1aY_1.
\]
But then $xy3x\cdot y3xy$ occurs in $g^\omega(0)$ through $x\mapsto X_1a,$ $y\mapsto Y_1,$ unless $Y_1=\varepsilon.$  If $Y_1=\varepsilon,$ then $Y'=\varepsilon$ and $Y$ has length $1$, so $X$ has length at least $2$.  But then $xy3x\cdot y3xy$ occurs in $g^\omega(0)$ through $x\mapsto X_1,$ $y\mapsto aY_1$.  \qedhere
\end{case}  
\end{proof}

\begin{lemma}\label{hAvoid}
The word $h(g^\omega(0))$ avoids $xyx\cdot yxy$.
\end{lemma}

\begin{proof}
Suppose towards a contradiction that $xyx\cdot yxy$ occurs in $h(g^\omega(0))$ through nonerasing morphism $f$.  Let $X=f(x)$ and $Y=f(y)$.  By exhaustive search of $h(g^8(0)),$ which contains all factors of $h(g^\omega(0))$ of length at most $3,$ we may assume that either $X$ or $Y$ has length at least $2$.  By an analysis similar to the one done in Lemma~\ref{hAvoid3}, the factors $XYX$ and $YXY$ have $h$-preimages that make up an occurrence of $xyx\cdot yxy$ in $g^\omega(0)$, contradicting Lemma~\ref{Clark}\ref{xyxdotyxy}.
\end{proof}

We are now ready to prove that $xyzx\cdot yzxy\cdot \rev{z}$ and $xyzx\cdot y\rev{z}xy$ are avoided by $h(g^\omega(0))$.  The techniques used in the proofs are similar to those used in the previous section to demonstrate the $4$-avoidability of the $3$-minimal formulas that flatten to $xyzyx\cdot zyxyz$.  Again, we note that once we know that there are only finitely many reversible factors in $h(g^\omega(0))$, the proof could be completed by carrying out Cassaigne's algorithm on a finite list of formulas with letters, but we have opted instead to write the proofs by hand.

\begin{theorem}\label{hmain1}
The formula $xyzx\cdot yzxy\cdot \rev{z}$ is avoided by $h\left(g^\omega(0)\right)$.
\end{theorem}

\begin{proof}
Suppose towards a contradiction that $xyzx\cdot yzxy\cdot \rev{z}$ occurs in $h(g^\omega(0))$ through morphism respecting reversal $f$.  Let $X=f(x),$ $Y=f(y),$ and $Z=f(y)$.  The only candidates for $Z$ are the reversible factors of $h(g^\omega(0))$.  By an exhaustive search of $h(g^8(0))$, in which all factors of $h(g^\omega(0))$ of length at most $4$ appear, these are exactly the factors of
\[
00,\ 11,\ 22,\ 030,\ 131,\ \mbox{and } 232. 
\]
We proceed with case work, showing that the appearance of factors $XYZX$ and $YZXY$ in $h(g^\omega(0))$ leads to a contradiction in each case.  We may assume that either $X$ or $Y$ has length at least $2$ (so that $XY$ has length at least $3$), as otherwise the $f$-image of any fragment of $xyzx\cdot yzxy\cdot\rev{z}$ has length at most $6$, and we have verified that this is impossible by exhaustive search.

\begin{case}
\item[$Z=3$] The appearance of factors $XY3X$ and $Y3XY$ contradicts Lemma~\ref{hAvoid3} immediately.

\item[$Z=b$, $b\in\{0,1,2\}$]
Consider the factors $XYbX$ and $YbXY.$  By Lemma~\ref{UniqueParsing}, $YbX$ parses uniquely into code words, so we either have $XY|bX$ and $Y|bXY$, or $XYb|X$ and $Yb|XY.$  

\begin{description}
\item[\textbf{Subcase a:} $XY|bX$ and $Y|bXY$] Since $|bXY$ appears in the second factor, by Lemma~\ref{Extend} \ref{Prefix3}, $XY$ is always preceded by $b$.  Therefore, the factors $XYbX$ and $YbXY$ appear internally as
\[
\underbrace{bX}_{\overset{\upmapsto}{x}}
\underbrace{Y}_{\overset{\upmapsto}{y}}
\underbrace{bX}_{\overset{\upmapsto}{x}}
\mbox{ and } 
\underbrace{Y}_{\overset{\upmapsto}{y}}
\underbrace{bX}_{\overset{\upmapsto}{x}}
\underbrace{Y}_{\overset{\upmapsto}{y}},
\] 
contradicting Lemma~\ref{hAvoid} as indicated.

\item[\textbf{Subcase b:} $XYb|X$ and $Yb|XY$] Since $XYb|$ appears in the first factor, by Lemma~\ref{Extend} \ref{Suffix3}, $XY$ is always followed by $b$.  Therefore, the factors $XYbX$ and $YbXY$ appear internally as
\[
\underbrace{X}_{\overset{\upmapsto}{x}}
\underbrace{Yb}_{\overset{\upmapsto}{y}}
\underbrace{X}_{\overset{\upmapsto}{x}}
\mbox{ and } 
\underbrace{Yb}_{\overset{\upmapsto}{y}}
\underbrace{X}_{\overset{\upmapsto}{x}}
\underbrace{Yb}_{\overset{\upmapsto}{y}},
\]
contradicting Lemma~\ref{hAvoid} as indicated.

\end{description}

\item[$Z=bb$, $b\in\{0,1,2\}$]
Consider the factors $XYb|bX$ and $Yb|bXY.$ Since $XYb|$ appears in the first factor and $|bXY$ appears in the second, by Lemma~\ref{Extend} parts~\ref{Suffix3} and~\ref{Prefix3}, every appearance of $XY$ is both followed and preceded by $b$.  But then $XYbbX$ and $YbbXY$ appear internally as
\[
\underbrace{bX}_{\overset{\upmapsto}{x}}
\underbrace{Yb}_{\overset{\upmapsto}{y}}
\underbrace{bX}_{\overset{\upmapsto}{x}}
\mbox{ and }
\underbrace{Yb}_{\overset{\upmapsto}{y}}
\underbrace{bX}_{\overset{\upmapsto}{x}}
\underbrace{Yb}_{\overset{\upmapsto}{y}},
\]
contradicting Lemma~\ref{hAvoid} as indicated.

\item[$Z=b3b$, $b\in\{0,1,2\}$]
Consider the factors $XYb|3|bX$ and $Yb|3|bXY$. Since $XYb|$ appears in the first factor and $|bXY$ appears in the second, by Lemma~\ref{Extend} parts~\ref{Suffix3} and~\ref{Prefix3}, every appearance of $XY$ is both followed and preceded by $b$.  But then $XYb3bX$ and $Yb3bXY$ appear internally as
\[
\underbrace{bX}_{\overset{\upmapsto}{x}}
\underbrace{Yb}_{\overset{\upmapsto}{y}}
3
\underbrace{bX}_{\overset{\upmapsto}{x}}
\mbox{ and }
\underbrace{Yb}_{\overset{\upmapsto}{y}}
3
\underbrace{bX}_{\overset{\upmapsto}{x}}
\underbrace{Yb}_{\overset{\upmapsto}{y}},
\]
contradicting Lemma~\ref{hAvoid3} as indicated.

\item[$Z=b3$, $b\in\{0,1,2\}$]
Consider the factors $XYb|3|X$ and $Yb|3|XY$.  Since $XYb|$ appears in the first factor, by Lemma~\ref{Extend}\ref{Suffix3}, every appearance of $XY$ is followed by $b$.  But then the factors $XYb3X$ and $Yb3XY$ appear internally as
\[
\underbrace{X}_{\overset{\upmapsto}{x}}
\underbrace{Yb}_{\overset{\upmapsto}{y}}
3
\underbrace{X}_{\overset{\upmapsto}{x}}
\mbox{ and }
\underbrace{Yb}_{\overset{\upmapsto}{y}}
3
\underbrace{X}_{\overset{\upmapsto}{x}}
\underbrace{Yb}_{\overset{\upmapsto}{y}},
\]
contradicting Lemma~\ref{hAvoid3} as indicated.

\item[$Z=3b$, $b\in\{0,1,2\}$]
Consider the factors $XY|3|bX$ and $Y|3|bXY$.  Since $|bXY$ appears in the second factor, by Lemma~\ref{Extend}\ref{Prefix3}, every appearance of $XY$ is preceded by $b$.  But then the factors $XY3bX$ and $Y3bXY$ appear internally as
\[
\underbrace{bX}_{\overset{\upmapsto}{x}}
\underbrace{Y}_{\overset{\upmapsto}{y}}
3
\underbrace{bX}_{\overset{\upmapsto}{x}}
\mbox{ and }
\underbrace{Y}_{\overset{\upmapsto}{y}}
3
\underbrace{bX}_{\overset{\upmapsto}{x}}
\underbrace{Y}_{\overset{\upmapsto}{y}},
\]
contradicting Lemma~\ref{hAvoid3} as indicated.  \qedhere
\end{case}
\end{proof}

\begin{theorem}
The formula $xyzx\cdot y\rev{z}xy$ is avoided by $h(g^\omega(0))$.
\end{theorem}

\begin{proof}
Suppose towards a contradiction that $xyzx\cdot y\rev{z}xy$ occurs in $h(g^\omega(0))$ through morphism respecting reversal $f$.  Let $X=f(x),$ $Y=f(y),$ and $Z=f(z).$  If $Z$ is palindromic then $Z=\rev{Z}$, and the arguments used in Theorem~\ref{hmain1} apply.  So $Z$ has the form $b3$ or $3b$ for some $b\in\{0,1,2\}$.  We may assume that either $X$ or $Y$ has length at least $3$, as otherwise the $f$-image of any fragment of $xyzx\cdot y\rev{z}xy$ has length at most $8$, and we have verified that this is impossible by exhaustive search.

Whether $Z=b3$ or $Z=3b,$ the factors $XYZX$ and $Y\rev{Z}XY$ contain the factors $Yb3$ and $Y3$, and $3bX$ and $3X$.  The factor $Yb3$ indicates that $Y=Y'\vert a$ in this factor (where $a+1\equiv b\pmod{3}$) and if $|Y|\geq 3$, then every appearance of $Y$ must parse this way by Lemma~\ref{UniqueParsing}.  But the factor $Y3$ makes this impossible.  So we may assume that $|Y|<3$.  However, then $|X|\geq 3$ and a similar argument applies since the factors $3bX$ and $3X$ both appear.
\end{proof}

\section{Conclusion}

In this article, we found $n$-avoidance bases $\Phi_1$, $\Phi_2,$ and $\Phi_3$ for formulas with reversal for $n=1$, $n=2$, and $n=3,$ respectively.  We found bounds on the avoidability index of every element in these bases, from which we are able to conclude that $xx$ is the unique element in $\Phi_1$ of highest avoidability index $3$, $xyx\cdot \rev{y}$ is the unique element in $\Phi_2$ of highest avoidability index $4$, and $xyzx\cdot\rev{y}\cdot\rev{z}$ is the unique element in $\Phi_3$ of highest avoidability index $5$.  Given that the formula with reversal
\[
\psi_n=xy_1y_2\dots y_nx\cdot \rev{y_1}\cdot \rev{y_2}\cdot\ldots\cdot\rev{y_n}
\]
is known to have avoidability index at least $4$ for all $n\geq 1$ \cite{CMR2016}, it would be interesting to know whether $\psi_n$ has highest avoidability index among all formulas with reversal on $n+1$ letters for each $n\geq 1$.  Since there are known constant bounds on the avoidability index of $\psi_n$\cite{CMR2016}, we suspect that this is not the case.  However, it is remarkable that these simple formulas with reversal $xx$, $xyx\cdot \rev{y}$, and $xyzx\cdot \rev{y}\cdot \rev{z}$ have higher avoidability index than any other minimal formula with reversal on the same number of variables.

While the exact avoidability indices of all elements in $\Phi_1$ and $\Phi_2$ are known, there are several formulas with reversal in $\Phi_3$ whose exact avoidability indices are unknown.  We suspect that the avoidability index of each of the formulas with reversal studied in Section~\ref{xyzyxSection} is $2$, and that the avoidability index of each of the formulas with reversal studied in Section~\ref{xyzxSection} is $3$.

Finally, we note that there are two main obstacles to finding a $4$-avoidance basis for formulas with reversal.  First of all, we do not have a nice characterization of avoidable formulas with reversal on $4$ variables (those with exactly three one-way variables are the issue).  Secondly, if we were to employ a computer check as we did in this article, the computation could be incredibly time-consuming, even with significant results analogous to Lemma~\ref{ThreeTwoWay}, Lemma~\ref{TwoTwoWay}, and Lemma~\ref{OneTwoWay}, which might reduce the length of the check.  Finally, we point out that a $4$-avoidance basis for classical formulas is still not known; this might be a more tractable intermediate problem.

\section*{Acknowledgements}

We would like to thank the anonymous reviewers for their helpful comments and suggestions.  This work was supported by the Natural Sciences and Engineering Research Council of Canada (NSERC), grant numbers 418646-2012 and 42410-2010.


\providecommand{\bysame}{\leavevmode\hbox to3em{\hrulefill}\thinspace}
\providecommand{\MR}{\relax\ifhmode\unskip\space\fi MR }
\providecommand{\MRhref}[2]{%
  \href{http://www.ams.org/mathscinet-getitem?mr=#1}{#2}
}
\providecommand{\href}[2]{#2}

\end{document}